\documentclass{amsart}
\usepackage{hyperref}

\newcommand{\R}{\mathbf{R}}

\newcommand{\N}{\mathbf{N}}
\newcommand{\tto}{\Rightarrow}

\newcommand{\eps}{\epsilon}

\DeclareMathOperator{\E}{\mathbf{E}}
\let\P\relax
\DeclareMathOperator{\P}{\mathbf{P}}

\theoremstyle{plain}
\newtheorem{theorem}{Theorem}

\newtheorem{lemma}[theorem]{Lemma}

\theoremstyle{definition}

\theoremstyle{remark}

\begin{document}
\title[CLT from the de Moivre--Laplace theorem]{Deriving the Central Limit Theorem from the De Moivre--Laplace Theorem}
\author{Calvin Wooyoung Chin}
\date{\today}
\begin{abstract}
The de Moivre--Laplace theorem is a special case of the central limit theorem for Bernoulli random variables, and can be proved by direct computation.
We deduce the central limit theorem for any random variable with finite variance from the de Moivre--Laplace theorem.
Our proof does not use advanced notions such as characteristic functions, the Brownian motion, or stopping times.
\end{abstract}
\maketitle

\section{Introduction}

Let $Z$ be a standard normal random variable; that is, assume
\begin{equation} \label{eq:normal_density}
\P(Z \le x) = \int_{-\infty}^x \frac{e^{-t^2/2}}{\sqrt{2\pi}} \,dt
\qquad \text{for all $x \in \R$.}
\end{equation}
If $S_1,S_2,\ldots$ are random variables and $W$ is a normal random variable, we write $S_n \tto W$ to mean
\[ \lim_{n\to\infty} \P(S_n \le x) = \P(W \le x) \qquad \text{for all $x \in \R$.} \]
In this case, $S_n$ is said to \emph{converge in distribution} to $W$.

The de Moivre--Laplace theorem, first published in 1738 \cite{DeM01} in a weak form, states that the binomial distribution may be approximated by the normal distribution.

\begin{theorem}[de Moivre--Laplace] \label{thm:dl}
Let $p \in (0,1)$, and $X_1,X_2,\ldots$ be independent and identically distributed random variables satisfying
\[ \P(X_1 = 1) = p \quad\text{and}\quad \P(X_1 = 0) = 1-p. \]
Then,
\[
\frac{X_1+\cdots+X_n - np}{\sqrt{np(1-p)}} \tto Z.
\]
\end{theorem}

Theorem~\ref{thm:dl} can be proved by direct computation using Stirling's formula
\[
\lim_{n\to\infty} \frac{n!}{\sqrt{2\pi n}(n/e)^n} = 1;
\]
see \cite[Section 7.3]{Chu12} for details.
The computation ``discovers" the formula \eqref{eq:normal_density} for the normal density without prior knowledge.
Thus, Theorem~\ref{thm:dl} might be used to motivate the introduction of the normal distribution.

Historically, Theorem~\ref{thm:dl} is the earliest instance of the celebrated Lindeberg--L\'evy central limit theorem (CLT).

\begin{theorem}[central limit theorem] \label{thm:clt}
If $X_1,X_2,\ldots$ are independent and identically distributed random variables satisfying $\E[X_1] = 0$ and $\E[X_1^2] = 1$, then
\begin{equation} \label{eq:conv_z}
\frac{X_1+\cdots+X_n}{\sqrt{n}} \tto Z.
\end{equation}
\end{theorem}

To see that Theorem~\ref{thm:dl} follows from Theorem~\ref{thm:clt}, note that $(X_1-p)/\sqrt{p(1-p)}$ in Theorem~\ref{thm:dl} has mean zero and variance one.
It is surprising that the convergence \eqref{eq:conv_z} is a \emph{universal} phenomenon observed from all random variables $X_1$ with finite variances, not just the Bernoulli ones.
This explains the ubiquity of the ``bell-shaped" curve in the real world.

Many proofs of Theorem~\ref{thm:clt} are known.
The standard proof \cite[Theorem~3.4.1]{Dur19} uses \emph{characteristic functions}, which are essentially Fourier transforms.
A related proof \cite[Subsection~2.2.3]{Tao12} uses the \emph{moment method}.

There have been efforts to find elementary proofs of Theorem~\ref{thm:clt} that avoid characteristic functions.
For instance, \cite{Tro59} replaced characteristic functions with certain linear operators on a function space.
Another proof using the \emph{Stein method} \cite{Ste86} makes use of the identity
\[ \E[f'(Z) - Zf(Z)] = 0 \]
that holds for certain well-behaved $f$.

The proof by \emph{Lindeberg swapping} \cite{Lin22} lets $Y_1,Y_2,\ldots$ to be independent standard normal random variables, and transform $(X_1+\cdots+X_n)/\sqrt{n}$ into $(Y_1+\cdots+Y_n)/\sqrt{n}$ by swapping $X_i$ for $Y_i$ one at a time.
Note that once we finish swapping, we are left with a standard normal random variable.

The last proof we would like to mention uses the \emph{Skorokhod representation}.
If $(B_t)_{t\ge0}$ is a Brownian motion, the Skorokhod representation theorem \cite[Theorem~8.1.1]{Dur19} provides us with stopping times $T_1\le T_2\le \cdots$ such that $X_1+\cdots+X_n$ has the same distribution as $B_{T_n}$.
The conclusion of the CLT follows from $T_n/n \to 1$ in probability.

In this note, we prove the CLT (Theorem~\ref{thm:clt}) in a new way, by deriving it directly from de Moivre--Laplace theorem (Theorem~\ref{thm:dl}).
In some sense, our proof is of similar spirit as the proof by Skorokhod embedding.
However, our proof is more elementary in that, for example, we do not need to construct the Brownian motion and build the theory of stopping times.

Section~\ref{sec:simple} is the crux of our proof.
There we prove the CLT when $X_1$ is \emph{simple}, that is, when it has only finitely many possible values.
Any proof of the CLT needs to deal with some measure theory, and Section~\ref{sec:gen} does that.
There we generalize the result to all $X_1$ with finite variance.
This step does not contain new idea, and might be considered routine in the eyes of experts.

\section{CLT for Simple Random Variables} \label{sec:simple}

Let us call a random variable with at mosts two possible values \emph{two-valued}.
De Moivre--Laplace theorem tells us that the CLT (Theorem~\ref{thm:clt}) holds when $X_1$ is two-valued.
Our task is to extend this result to when $X_1$ takes more than two values.

We say that a random variable $Y$ is a \emph{finite mixture} of random variables $Y_1,\ldots,Y_n$ ($n \in \N$) if there is a random number $\theta \in \{1,\ldots,n\}$ independent from $Y_1,\ldots,Y_n$ such that $Y = Y_i$ on the event $\{\theta = i\}$.
In short, we write $Y = Y_\theta$.

\begin{lemma} \label{lem:mixture}
Any simple mean-zero random variable $X$ has a distribution of a finite mixture of two-valued mean-zero random variables.
\end{lemma}

It would be instructive to try proving this before reading the proof.

\begin{proof}
If $X\equiv 0$, we can let $n=1$, $Y_1 \equiv 0$, and $\theta \equiv 1$.
Otherwise, let $a, b > 0$ be the smallest numbers such that $\P(X=a) > 0$ and $\P(X=-b) > 0$.
We proceed by induction on the number of possible values of $X$.
If
\[ a\P(X=a) \le b\P(X=-b), \]
choose an event $B \subset \{X=-b\}$ such that $a\P(X=a) = b\P(B)$.
(If needed, we may let $[0,1)$ be the underlying probability space.)
Notice that $X$ is two-valued and mean-zero if conditioned on $\{X=a\} \cup B$.
On the other hand, conditioned on $(\{X=a\} \cup B)^c$, we see that $X$ has less number of possible values, so by the induction hypothesis, it has a distribution of a finite mixture of two-valued mean-zero random variables.
Combining the two parts, we conclude that $X$ also has a distribution of a finite mixture of two-valued mean-zero random variables.
An analogous proof works when
\[ a\P(X=a) > b\P(X=-b). \qedhere \]
\end{proof}

\begin{lemma}[CLT for simple random variables] \label{lem:clt_simple}
The CLT (Theorem~\ref{thm:clt}) holds if $X_1$ is simple.
\end{lemma}

\begin{proof}
By Lemma~\ref{lem:mixture}, we may assume $X_n = Y_{n,\theta_n}$ where
\[ \{Y_{n,i}; n\in \N, i=1,\ldots,m\} \]
($m\in\N$) is an independent family of two-valued mean-zero random variables,
$Y_{1,i},Y_{2,i},\ldots$ are identically distributed for $i=1,\ldots,m$, and $\theta_1,\theta_2,\ldots \in \{1,\ldots,m\}$ are i.i.d.\ and independent from $Y$'s. We may assume $\P(\theta_1=i) > 0$ for $i=1,\ldots,m$.

Let $x \in \R$. Intuitively, we have
\[
\P\Bigl(\frac{X_1+\cdots+X_n}{\sqrt{n}} \le x \Bigr)
= \E\biggl[\P\Bigl(\frac{Y_{1,\theta_1}+\cdots+Y_{n,\theta_n}}{\sqrt{n}} \le x \Bigm| \theta_1,\ldots,\theta_n \Bigr) \biggr].
\]
We instead choose the following more concrete formulation:
\begin{equation} \label{eq:clt_cond}
\P\Bigl(\frac{X_1+\cdots+X_n}{\sqrt{n}} \le x \Bigr)
= \E\biggl[\P\Bigl(\frac{Y_{1,i_1}+\cdots+Y_{n,i_n}}{\sqrt{n}} \le x\Bigr)_{(i_1,\ldots,i_n) = (\theta_1,\ldots,\theta_n)}\biggr]
\end{equation}
where $p(i_1,\ldots,i_n)_{(i_1,\ldots,i_n)=(\theta_1,\ldots,\theta_n)}$ denotes $p(\theta_1,\ldots,\theta_n)$ for any function $p$.
(See \cite[Theorem~20.3]{Bil12}.)

If $i_1,i_2,\ldots \in \{1,\ldots,m\}$ satisfy
\[ \lim_{n\to\infty} \frac{\bigl| \{k : 1\le k \le n \text{ and } i_k = i\} \bigr|}{n} = \P(\theta_1 = i) \]
for $i=1,\ldots,m$, then
\[
\frac{1}{\sqrt{n}} \sum_{\substack{k=1,\ldots,n \\ i_k = i}} Y_{k,i_k}
= \frac{\sqrt{\E[Y_{1,i}^2]\P(\theta_1=i)}}{\sqrt{\E[Y_{1,i}^2]\P(\theta_1=i)n}} \sum_{\substack{k=1,\ldots,n \\ i_k = i}} Y_{k,i}
\tto \sqrt{\E[Y_{1,i}^2]\P(\theta_1 = i)} Z
\]
for $i=1,\ldots,m$ by the de Moivre--Laplace theorem (Theorem~\ref{thm:dl}).
Summing up for $i=1,\ldots,m$, we have
\[
\frac{Y_{1,i_1} + \cdots + Y_{n,i_n}}{\sqrt{n}} \tto \sqrt{\sum_{i=1}^m \E[Y_{1,i}^2]\P(\theta_1=i)} \cdot Z = Z.
\]

By (a relatively easily proved version \cite[Theorem~6.1]{Bil12} of) the strong law of large numbers, we have
\[
\lim_{n\to\infty} \frac{\bigl| \{k : 1\le k \le n \text{ and } \theta_k = i\} \bigr|}{n} = \P(\theta_1 = i)
\qquad \text{with probability $1$}
\]
for $i=1,\ldots,m$.
Thus,
\[
\lim_{n\to\infty}
\P\Bigl(\frac{Y_{1,i_1}+\cdots+Y_{n,i_n}}{\sqrt{n}} \le x\Bigr)_{(i_1,\ldots,i_n) = (\theta_1,\ldots,\theta_n)} = \P(Z\le x)
\]
with probability $1$.
By \eqref{eq:clt_cond}, it follows that
\begin{equation} \label{eq:conv_df_z}
\lim_{n\to\infty} \P\Bigl(\frac{X_1+\cdots+X_n}{\sqrt{n}} \le x \Bigr)
= \P(Z \le x).
\end{equation}
Since $x$ is arbitrary, we have \eqref{eq:conv_z}.
\end{proof}

\section{Generalization} \label{sec:gen}

We are left with a measure-theoretic argument which is rather standard.
A similar argument can be found in \cite[Subsection~2.2.1]{Tao12}.

\newcommand{\eY}{Y^{(\eta)}}

\begin{proof}[Proof of Theorem~\ref{thm:clt}]
Let $x \in \R$ be given, and let us show \eqref{eq:conv_df_z}.
For brevity, let $S_n := (X_1+\cdots+X_n)/\sqrt{n}$.
Let $\eps > 0$ be given, and let $\delta > 0$ be such that
\[ \P(Z \le x+\delta) < \P(Z \le x) + \eps. \]

Given $\eta > 0$, we can take an i.i.d.\ sequence $\eY_1,\eY_2,\ldots$ of simple random variables such that $\E[\eY_1] = 0$, $\E[(\eY_1)^2] = 1$, and
\begin{equation} \label{eq:ey_close}
\E[(X_i - \eY_i)^2] \le \eta \qquad \text{for $i=1,2,\dots$.}
\end{equation}
Let us write $T_n := (\eY_1+\cdots+\eY_n)/\sqrt{n}$.
By Lemma~\ref{lem:clt_simple},
\[ \lim_{n\to\infty} \P(T_n \le x + \delta) = \P(Z \le x+\delta)
< \P(Z \le x) + \eps. \]

By Chebyshev's inequality and \eqref{eq:ey_close}, we have
\[
\P(|S_n-T_n|>\delta) \le
\frac{1}{\delta^{2}n}\sum_{i=1}^n\E[(X_i-Y^{(\eta)}_i)^2]
\le \frac{\eta}{\delta^2}.
\]
If we let $\eta := \delta^2\eps$, the right side is $\eps$, and thus
\[
\begin{split}
\P(S_n \le x) &\le \P(|S_n-T_n| > \delta) + \P(T_n \le x+\delta) \\
&< \P(Z\le x) + 2\eps
\end{split}
\]
for all large $n$.
By a similar argument, once can prove that
\[ \P(S_n\le x) > \P(Z\le x) - 2\eps \]
for all large $n$. Since $\eps$ is arbitrary, we have \eqref{eq:conv_df_z}.
\end{proof}

\section*{Acknowledgements}
The author is supported in part by the National Research Foundation of Korea grants 2017R1A2B2001952 and 2019R1A5A1028324.
The author would like to thank Yuval Peres for helpful suggestions.


\end{document}